\documentclass[11pt,reqno]{amsart}
\usepackage{tikz}
\usetikzlibrary{shapes.geometric}
\usepackage[center]{caption}
\usepackage{xcolor}

\usepackage{amssymb,latexsym,amsmath,mathtools}
\usepackage{graphicx,mathrsfs}
\usepackage{subfigure}
\usepackage{float}
\usepackage{hyperref,url}
\hypersetup{colorlinks=true, urlcolor=black, linkcolor=black, citecolor=black}
\usepackage{pict2e}
\usepackage{amstext}
\usepackage{bbm}
\usepackage{amsthm}
\usepackage{dsfont}
\usepackage{amsfonts}
\usepackage{comment}

\newcommand{\sub}{\subseteq}
\newcommand{\Z}{\mathbb{Z}}
\newcommand{\R}{\mathbb{R}}
\newcommand{\C}{\mathbb{C}}

\newcommand{\N}{\mathbb{N}}
\newcommand{\eps}{\varepsilon}

\newcommand{\fj}[2]{\left\lfloor #1 \right\rfloor_#2}

\newtheorem{theorem}{Theorem}[section]
\newtheorem{lemma}[theorem]{Lemma}

\newtheorem{proposition}[theorem]{Proposition}

\newtheorem{corollary}[theorem]{Corollary}

\numberwithin{equation}{section}
\DeclarePairedDelimiter{\norm}{\lVert}{\rVert}

\makeatletter
\let\oldnorm\norm
\def\norm{\@ifstar{\oldnorm}{\oldnorm*}}

\makeatother

\begin{document}

\title[Curved Kakeya]
{Construction of a curved Kakeya set}

\author{Tongou Yang}

\address[Tongou Yang]{Department of Mathematics, University of California\\
Los Angeles, California 90095, United States}
\email{tongouyang@math.ucla.edu}

\author{Yue Zhong}

\address[Yue Zhong]{ Department of Mathematics, Sun Yat-sen University \\
Guangzhou, Guangdong 510275, P.R. China}
\email{zhongy69@mail3.sysu.edu.cn}

\dedicatory{} 
\commby{}


\begin{abstract}
We construct a compact set in $\R^2$ of measure $0$ containing a piece of a parabola of every aperture between $1$ and $2$. As a consequence, we improve lower bounds for the $L^p$-$L^q$ norm of the corresponding maximal operator for a range of $p,q$. Moreover, our construction can be generalised from parabolas to a family of $C^2$ curves satisfying suitable curvature conditions.
\end{abstract}


\maketitle

\tableofcontents


\section{Introduction}
Consider Wolff's circular maximal Kakeya function introduced in \cite{Kolasa_Wolff}:
\begin{equation*}
    \mathcal C f(r)=\sup_{x\in \R^2}\int_{C(x,r)}|f(y)|dy,
\end{equation*}
initially defined for continuous functions $f:\R^2\to \C$ with compact support, where $r\in [1,2]$ and $C(x,r)$ denotes the circle centred at $x$ of radius $r$. For $\delta>0$, we can also consider a $\delta$-thickened version of the maximal function, defined by
\begin{equation*}
    \mathcal C_\delta f(r)=\sup_{x\in \R^2}\frac 1{|C_\delta(x,r)|}\int_{C_\delta(x,r)}|f(y)|dy,
\end{equation*}
where $C_\delta(x,r)$ denotes the annulus centred at $x$ of radius $r$ and thickness $\delta$, namely, $C_\delta(x,r)=\{y\in \R^2:r-\delta\le |y-x|\le r+\delta\}$. For Lebesgue exponents $p,q\in [1,\infty]$, we are interested in the $L^p(\R^2)\to L^q([1,2])$ mapping property of $\mathcal C$ and $\mathcal C_\delta$. Wolff \cite{WolffKakeyaL3} proved the bound 
\begin{equation*}
    \norm{\mathcal C_\delta}_{L^3\to L^3}\lesssim_\eps \delta^{-\eps},\quad \forall \eps>0,
\end{equation*}
which is sharp except for the $\eps$-loss. Using this, he concluded that every compact subset of $\R^2$ containing a circle of every radius between $1$ and $2$ must have Hausdorff dimension $2$.

A closely related analogue of $\mathcal C$ is a parabolic maximal function defined by
\begin{equation*}
    \mathcal P f(a)=\sup_{(x_1,x_2)\in \R^2}\int_{0}^{1} |f(x_1+t,x_2+at^2)|dt.
\end{equation*}
Similarly, for each $\delta>0$, we can define the $\delta$-thickened version
\begin{equation*}
    \mathcal P_\delta f(a)=\sup_{(x_1,x_2)\in \R^2}(2\delta)^{-1}\int_{0}^{1}\int_{-\delta}^{\delta} |f(x_1+t,x_2+at^2+s)|ds dt.
\end{equation*}

\subsection{Curves of cinematic curvature}\label{sec:cinematic}
Both maximal functions can be thought of as special cases of a family of curves with {\it cinematic curvature}, introduced by Sogge \cite{Sogge_first_cinematic}. See also \cite{Kolasa_Wolff}, \cite{Zahl_algebraic}, \cite{Zahl_nonalgebraic}, \cite{PYZ2022}, \cite{ChenGuoYang2023}, \cite{Zahl2023}, \cite{ChenGuo} for related discussions on maximal operator bounds related to curves of cinematic curvature. For example, it can be checked that for any smooth function $f:[0,1]\to \R$ with $f''\ne 0$, the family of curves $\{(t,af(t)):0\le t\le 1\}$, $1\le a\le 2$ satisfies the cinematic curvature condition defined in \cite{Kolasa_Wolff}. Thus, one has the following upper bounds for the parabolic maximal operator:
    \begin{align}
        &\norm{\mathcal P_\delta f}_{L^3([1,2])}\lesssim_\eps \delta^{-\eps}\norm f_{L^3(\R^2)},\quad \forall \eps>0,\label{eqn_L3}\\
        &\norm{\mathcal P_\delta f}_{L^q([1,2])}\lesssim \delta^{\frac 1 2-\frac 3 {2p}}\norm{f}_{L^p(\R^2)},\quad p<\frac 8 3,\quad q\ge \frac{2p}{p-1}.\label{eqn_KW}        
    \end{align}
Indeed, \eqref{eqn_L3} follows from, say, \cite{Zahl_algebraic}, and \eqref{eqn_KW} follows from \cite{Kolasa_Wolff}. It is worth noting that \eqref{eqn_KW} has no $\eps$-losses in the power of $\delta$; on the other hand, we do not know whether the condition $p<8/3$ is necessary for \eqref{eqn_KW} to hold, although it is used in the combinatorial proof in \cite{Kolasa_Wolff}.

\subsection{Lower bounds}
In order to simplify the exposition, we first restrict ourselves to the parabolic maximal operator $\mathcal P_\delta$. Define 
\begin{equation*}
       B(p,q)=\sup_{f\ne 0}\frac {\norm{\mathcal P f}_{L^q([1,2])}}{\norm f_{L^p(\R^2)}},\quad B(p,q,\delta)=\sup_{f\ne 0}\frac {\norm{\mathcal P_\delta f}_{L^q([1,2])}}{\norm f_{L^p(\R^2)}}.
\end{equation*}
The inequalities \eqref{eqn_L3} and \eqref{eqn_KW} give upper bounds for $B(p,q,\delta)$. On the other hand, the existence of some curved Kakeya sets provides some quantitative lower bounds for $B(p,q,\delta)$. For clarity, we first use the notion of (vertical) $\delta$-thickening $S(\delta)$ for a subset $S\sub \R^2$:
\begin{equation*}
    S(\delta):=\{x+(0,\eps):x\in S,-\delta\le \eps\le \delta\}.
\end{equation*}
Also, throughout this article, we denote by $|S|$ the two-dimensional Lebesgue measure of a measurable subset $S\sub \R^2$. 
\begin{theorem}[Parabolic variant of Kolasa-Wolff construction \cite{Kolasa_Wolff}]\label{thm_KW}
There exists a compact subset $K$ of $\R^2$ of Lebesgue measure $0$ that contains a piece of length $\sim 1$ of a parabola of every aperture between $1$ and $2$. Moreover, its $\delta$-thickening $K(\delta)$ satisfies $|K(\delta)|\lesssim (\log \delta^{-1})^{-2}(\log \log \delta^{-1})^{2}$. Consequently, we have the lower bound for every $p,q\in [1,\infty]$:
\begin{equation*}
    B(p,q,\delta)\gtrsim (\log \delta^{-1})^{2/p}(\log \log \delta^{-1})^{-2/p},
\end{equation*}
where the implicit constant is independent of $\delta$. In particular, $B(p,q)=\infty$ if $p<\infty$.
\end{theorem}
Indeed, this follows from an easy adaptation of the main construction of \cite[Proposition 1.1]{Kolasa_Wolff} for circles to the case of parabolas.

We also refer to other related Kakeya set constructions, such as those in \cite{Schoenberg1}, \cite{Schoenberg2}, \cite{BesicovitchRado}, \cite{Kinney}, \cite{Davies}, \cite{Talagrand}, \cite{Sawyer}, \cite{Keich}, \cite{ChangCsornyei}, \cite{MR4534746}, \cite{ChenYanZhong}.

The main theorem of this paper is an improvement of Theorem \ref{thm_KW} as follows.
\begin{theorem}[Main theorem]\label{thm_main}
There exists a compact subset $K$ of $\R^2$ of Lebesgue measure $0$ that contains a piece of length $\sim 1$ of a parabola of every aperture between $1$ and $2$. Moreover, its $\delta$-thickening $K(\delta)$ has measure $\lesssim (\log \delta^{-1})^{-2}$. Consequently, we have the lower bound 
\begin{equation*}
    B(p,q,\delta)\gtrsim (\log \delta^{-1})^{2/p}.
\end{equation*}
\end{theorem}
Namely, by refining the main construction in \cite{Kolasa_Wolff}, Theorem \ref{thm_main} removes the $\log \log \delta^{-1}$ factor in the lower bound given by Theorem \ref{thm_KW}.

\subsection{A generalised curved Kakeya set}
We can generalise the construction in Theorem \ref{thm_main} from parabolas to graphs of functions obeying a suitable curvature condition (\eqref{eqn_condition_f_3} below).

Define a $C^2$-function $f:[0,1]\to \R$ that satisfies the following conditions:
\begin{align}
& f'(0)\ge 0,\quad \inf_{x\in [0,1]} f''(x)>0, \label{eqn_condition_f_1}\\
    & f''' \text{ exists and is bounded on }(0,1),\label{eqn_condition_f_2}\\
    & f'f'''-(f'')^2\le 0 \text{ on }(0,1).\label{eqn_condition_f_3}
\end{align}
With this, define the corresponding maximal operators:
\begin{align}
    \mathcal R g(a)&=\sup_{(x_1,x_2)\in \R^2}\int_{0}^{1} |g(x_1+t,x_2+af(t))| dt\nonumber\\
    \mathcal R_\delta g(a)&=\sup_{(x_1,x_2)\in \R^2}(2\delta)^{-1}\int_{0}^{1}\int_{-\delta}^{\delta} |g(x_1+t,x_2+af(t)+s)|ds dt,\label{eqn:cinematic_delta}
\end{align}  
and denote the corresponding operator norms
\begin{equation*}
       R(p,q)=\sup_{g\ne 0}\frac {\norm{\mathcal R g}_{L^q([1,2])}}{\norm g_{L^p(\R^2)}},\quad R(p,q,\delta)=\sup_{g\ne 0}\frac {\norm{\mathcal R_\delta g}_{L^q([1,2])}}{\norm g_{L^p(\R^2)}}.
\end{equation*}

\begin{theorem}\label{thm_general}
    Let $f:[0,1]\to \R$ be a $C^2$ function obeying \eqref{eqn_condition_f_1}\eqref{eqn_condition_f_2}\eqref{eqn_condition_f_3}. Then there exists a compact subset $K$ of $\R^2$ of Lebesgue measure $0$ that contains a translated copy of a piece of length $\sim 1$ of the graph of $y=af(x)$ for every $1\le a\le 2$. Moreover, its $\delta$-thickening $K(\delta)$ has measure $\lesssim (\log \delta^{-1})^{-2}$. Consequently, we have the lower bound 
\begin{equation}\label{eqn_M_lower_bound}
    R(p,q,\delta)\gtrsim (\log \delta^{-1})^{2/p},
\end{equation}
and in particular, $R(p,q)=\infty$ for $p<\infty$.
\end{theorem}
In the following of this article, we prove Theorem \ref{thm_general}, from which Theorem \ref{thm_main} follows as a corollary.

\subsection{Outline of the article}
In Section \ref{sec_construction} we construct a thickened version of the curved Kakeya set $K$. In Section \ref{sec_iteration} we iterate the preceding construction to obtain an actual curved Kakeya set of zero measure. In the Appendix, we give a brief summary of the current best upper and lower bounds for $R(p,q,\delta)$.

\subsection{Acknowledgements}
Tongou Yang was supported by the Croucher Fellowships for Postdoctoral Research. Yue Zhong was supported in part by the National Key R\&D Program of China (No. 2022YFA1005700) and the NNSF of China (No. 12371105). Both authors thank Sanghyuk Lee and Shaoming Guo for bringing this problem to our attention, and Lixin Yan, Xianghong Chen and Mingfeng Chen for helpful suggestions. We especially thank an anonymous referee for their careful reading and very detailed suggestions, which greatly improved the presentation of the manuscript.


\section{Compression by forcing tangencies}\label{sec_construction}

Let $M\in 2\N$ be large enough. In this section, we are going to present the $M$-th building blocks $F_M$ of the construction of the curved Kakeya set $K$ in Theorem \ref{thm_general}. This is done using a ``cut-and-slide" procedure. The idea at each step $j$ is to create many tangencies at a fixed $x$-coordinate $x_j$ by translations, so that the curved rectangles are compressed near $x_j$. The following is the precise construction.

Unless otherwise specified, all implicit constants are allowed to depend on $f$ only; more precisely, they depend on $\norm{f}_{C^2}$, $\inf f''$ and $\sup |f'''|$ given in \eqref{eqn_condition_f_1}\eqref{eqn_condition_f_2}\eqref{eqn_condition_f_3}.

We start with a $C^2$-function $f:[0,1]\to \R$ obeying \eqref{eqn_condition_f_1}\eqref{eqn_condition_f_2}\eqref{eqn_condition_f_3}. Fix $1 \leq a_0 \leq 2 $ and $0 < \delta_0 \leq 2-a_0$. Our goal in this section is to construct $F_M=F_M(a_0,\delta_0)$ as follows.

\subsection{Step 0}
Consider the initial ``curved rectangle" 
\begin{equation}\label{eqn_T0}
    T^{(0)}(a_0,\delta_0):=\{(x,af(x)):x\in [0,1],a\in [a_0,a_0+\delta_0]\}.
\end{equation}
Fix $M\in 2\mathbb N$. We divide $T^{(0)}(a_0,\delta_0)$ into $2^M$ ``curved rectangles" of the form
\begin{equation*}
    T^{(0)}_{n}:=T^{(0)}(a_n,\delta_0 2^{-M}),\quad n=0,1,\dots,2^M-1,
\end{equation*}
where we denote
\begin{equation}\label{eqn_defn_an}
    a_n:=a_0+n\delta_0 2^{-M}.
\end{equation}
The curve at the bottom of $T^{(0)}_{n}$ is given by
\begin{equation*}
    L^{(0)}_n(x):=a_n f(x).
\end{equation*}

Partition $[0,1]$ uniformly into $M/2$ intervals of length $2M^{-1}$, where the partitioning points are given by 
\begin{equation}\label{eqn:defn_x_j}
    x_j:=2jM^{-1},\quad j=0,1,\dots,M/2.
\end{equation}

\subsection{Step 1}
For odd $n=1,3,\dots,2^{M}-1$, we now apply different translations to $T^{(0)}_{n}$ so that its bottom curve $L^{(0)}_{n}$ will be tangent to the bottom curve $L^{(0)}_{n-1}$ at $x_1$. That is, we need to find translations $u:=u_{n}^{(1)}$, $v:=v_{n}^{(1)}$ such that
$$
\left\{
\begin{aligned}
    & a_n f(x_1-u) + v = a_{n-1} f(x_1) \\
    & a_n f'(x_1-u) = a_{n-1} f'(x_1) \\
\end{aligned}  
\right.
$$

To find the solutions, we first focus on the second equation. First, using the implicit function theorem and the fact that $\inf f''>0$, we see that for $M$ large enough, such $u$ always exists and $0<u\lesssim \delta_0 2^{-M}$. To find the expression of $u$, by the mean value theorem, there exists some $\xi=\xi_{n}^{(1)}$ such that
\begin{equation*}
    f'(x_1)-f'(x_1-u)=uf''(\xi).
\end{equation*}
Moreover, such $\xi$ must be unique since $f''>0$. Thus the second equation gives 
\begin{equation*}
    u=u_n^{(1)}=\frac {\delta_0 2^{-M}}{a_n}\frac{f'(x_1)}{f''(\xi)}.
\end{equation*}
Plugging into the first equation, we can find $v$, which is also positive. Also, by Lemma \ref{lemma of lower} below, the translated curve $L_{n}^{(0)}+(u,v)$ is strictly above $L_{n-1}^{(0)}$ except at the tangent point.

After Step 1, we obtain $2^{M-1}$ larger curved figures
\begin{equation}\label{T^1}
    T^{(1)}_n:=T_n^{(0)}\cup (T_{n+1}^{(0)}+(u_{n+1}^{(1)},v_{n+1}^{(1)})),
\end{equation}
where $n=0,2,\dots,2^{M}-2$, that are compressed well at $x_1$. Moreover, for each even $n$, the curve $L_n^{(0)}$ is still the bottom curve of $T^{(1)}_n$. Denote
\begin{equation*}
    T^{(1)}:=\bigcup_{n\in 2\N} T^{(1)}_n.
\end{equation*}

\subsection{Step $2$}
We continue in a similar way, this time compressing $T^{(1)}_n$, $n=0,2,\dots,2^M-2$ at the point $x_2$. More precisely, for $n=2,6,10,\dots,2^M-2$, we translate $T^{(1)}_{n}$ further by some $(u_{n}^{(2)},v_{n}^{(2)})$ so that its bottom curve $L_{n}^{(0)}$ is tangent to $L^{(0)}_{n-2}$ at $x_2$. Similarly to \eqref{T^1}, we obtain $2^{M-2}$ larger curved figures $T^{(2)}_n$, $n=0,4,\dots,2^M-4$, whose bottom curve is $L_n^{(0)}$ by Lemma \ref{lemma of lower} below, and they are compressed well at $x_2$. Denote
\begin{equation*}
    T^{(2)}:=\bigcup_{n\in 4\N} T^{(2)}_n.
\end{equation*}
Refer to Figure \ref{fig:steps}, which shows two steps of translations when $2^M=16$. Here after Step 1, the $1,3,\dots,15$th pieces are translated so that their bottoms are tangent at $x_1$ to the bottoms of $0,2,\dots,14$th pieces, respectively. After Step 2, the 1st and the 2nd pieces (the pair $(1,2)$) are translated together, such that the bottom of the 2nd piece is tangent at $x_2$ to the bottom of the $0$th piece. Note that the 1st piece is translated in both Steps 1 and 2. The pairs $(5,6),(9,10),(13,14)$ are translated in a similar way to the pair $(1,2)$.
\begin{figure}
    \centering
    \includegraphics[width=\linewidth]{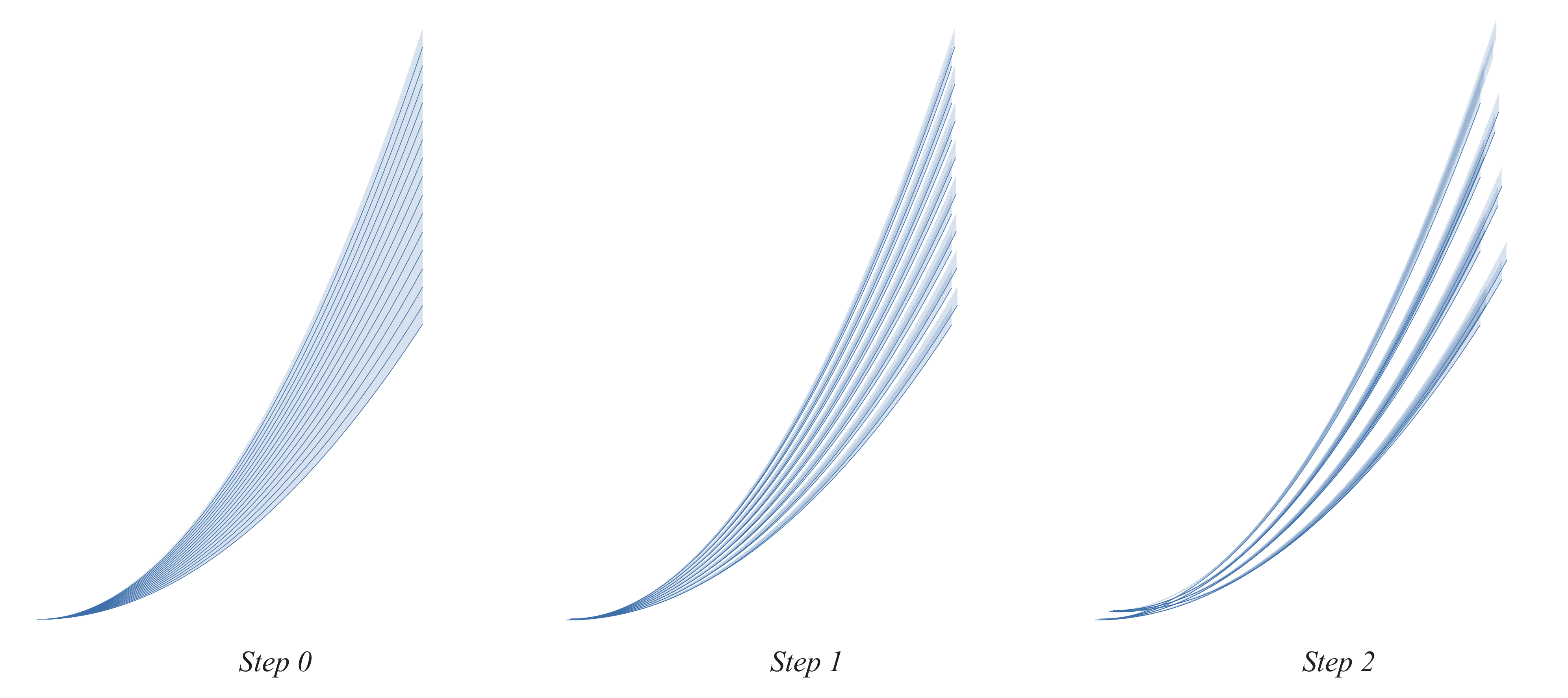}
    \caption{Figure after two steps when $2^M=16$.}
    \label{fig:steps}
\end{figure}

\subsection{Step $j$}
Now we describe a general Step $j$. For each $n$ of the form $2^j k+2^{j-1}$, $k=0,1,\dots,2^{M-j}-1$, we translate $T_n^{(j-1)}$ by some $(u^{(j)}_{n},v^{(j)}_{n})$ so that its bottom curve $L^{(0)}_{n}$ is tangent to $L_{n-2^{j-1}}^{(0)}$ at $x_j$. The system of equations we need to solve is now
\begin{equation}\label{eqn:general_step_j}
\left\{
\begin{aligned}
    & a_n f(x_j-u_n^{(j)}) + v_n^{(j)} = a_{n-2^{j-1}} f(x_j) \\
    & a_n f'(x_j-u_n^{(j)}) = a_{n-2^{j-1}} f'(x_j) \\
\end{aligned}  
\right.
\end{equation}

By the same computation as in Step 1, we have
\begin{equation}\label{eqn_unj}
    u_{n}^{(j)} = \frac {\delta_0 2^{j-1-M}}{a_n}\frac{f'( x_j)}{f''(\xi)},
\end{equation}
where $\xi$ is the unique number such that 
\begin{equation*}
    f''(\xi)=\frac{f'(x_j)-f'(x_j-u^{(j)}_{n})}{u^{(j)}_{n}},
\end{equation*}
and its existence is guaranteed by the implicit function theorem. Thus \eqref{eqn_unj} implies that
\begin{equation}\label{eqn:unj_bound}
    0<u_{n}^{(j)}\lesssim \delta_0 2^{j-M}.
\end{equation}
Using this, $|\xi-x_j|\le |u_n^{(j)}|$ and the fact that $f'''$ is bounded, we further record that
\begin{equation}\label{eqn:unj_big_O}
    u_{n}^{(j)} = \frac {\delta_0 2^{j-1-M}}{a_n}\frac{f'( x_j)}{f''(x_j)}+O(\delta_0^2  2^{2j-2M}).
\end{equation}
Now we come to $v_n^{(j)}$. By Taylor's theorem, the first equality of \eqref{eqn:general_step_j} gives that
\begin{equation}\label{eqn_vnj}
\begin{aligned}
    v_n^{(j)} &=a_{n-2^{j-1}} f(x_j)-a_n f(x_j-u_{n}^{(j)})\\
    &=a_{n-2^{j-1}} f(x_j)-a_n f(x_j)+a_n f'(x_j)u_{n}^{(j)}+O(u_{n}^{(j)})^2\\
    &=\delta_0 2^{j-1-M}\left(-f(x_j)+\frac{f'(x_j)^2}{f''(x_j)}\right)+O(\delta_0^2  2^{2j-2M}),
\end{aligned}   
\end{equation}
where we have also used \eqref{eqn_defn_an} and \eqref{eqn:unj_big_O}. In particular, we also have
\begin{equation}\label{eqn:vnj_bound}
    |v_n^{(j)}|\lesssim \delta_0 2^{j-M}.
\end{equation}

Thus, similarly to \eqref{T^1}, we obtain $2^{M-j}$ larger curved figures $T_n^{(j)}$, $n=0,2^j,\dots,2^M-2^j$, which are compressed well at $x_j$. Also, by the following lemma, the bottom of $T_n^{(j)}$ is still $L_n^{(0)}$. Denote
\begin{equation*}
    T^{(j)}:=\bigcup_{n\in 2^j \N} T^{(j)}_n.
\end{equation*}

\begin{lemma}\label{lemma of lower}
    If $f:[0,1]\to \R$ is a $C^2$ function obeying \eqref{eqn_condition_f_1}\eqref{eqn_condition_f_2}\eqref{eqn_condition_f_3}, $x_0\in [0,1]$, and $(u,v)$ is the solution of the equation 
    $
    \left\{
    \begin{aligned}
        & a f(x_0-u) + v = \tilde a f(x_0) \\
        & a f'(x_0-u) = \tilde a f'(x_0) \\
    \end{aligned},  
    \right.
    $
    then we have $a f(x-u) + v \geq \tilde a f(x)$ for any $x$ while $a\ge \tilde a$ (such that $u,v$ exist).
\end{lemma}
\begin{proof}
    Fix $\tilde a,x,x_0$. Let 
    \begin{equation*}
        F(a) = a(f(x-u(a)) - f(x_0 - u(a))) - \tilde a(f(x)-f(x_0)),
    \end{equation*}
    where we regard $u,v$ as functions of $a$. We have $F(\tilde a) = 0 $ since $u(\tilde a) =  v(\tilde a) = 0$.
    We want to show $F(a) \geq 0 $ while $a \geq \tilde a$, and it suffices to show that $F'(a) \geq 0$ for $a\ge \tilde a$.
    
    Since  $u$ is the solution of the equation $ a f'(x_0-u) = \tilde a f'(x_0) $, we have 
    $$\frac{\partial}{\partial a}(a f'(x_0-u)) = \frac{\partial}{\partial a}(\tilde a f'(x_0)),$$
    which means that 
    $$u'(a) = \frac{f'(x_0-u)}{a f^{''}(x_0-u)} .$$
    Denote $X = x - u$ and $X_0=x_0-u$. Then by direct computation,
    \begin{align*}
        F'(a) &= (f(X) - f(X_0)) - a(f'(X) - f'(X_0)) \cdot u'(a) \\
        &= \dfrac{(f(X) - f(X_0)) f''(X_0) - (f'(X) - f'(X_0))  f'(X_0)}{ f''(X_0)}\\
        &=\frac{G(X)-G(X_0)}{f''(X_0)},
    \end{align*}
    where we have denoted
    \begin{equation*}
        G(X):=f(X)f''(X_0)-f'(X)f'(X_0).
    \end{equation*}
    Then it suffices to show $G(X)\ge G(X_0)$ for all $X$, which is true if we can show $G'(X)\ge 0$ for $X\ge X_0$ and $G'(X)\le 0$ for $X\le X_0$. To this end, we consider
    \begin{equation*}
        H(X):=\frac{G'(X)}{f'(X)}=\frac{f'(X)f''(X_0)-f''(X)f'(X_0)}{f'(X)}.
    \end{equation*}
    We note that $X=x-u>0$ since $x\ge 2/M\ge u$. Thus by \eqref{eqn_condition_f_1}, $f'(X)>0$. Thus it suffices to show that $H(X)\ge 0$ for $X\ge X_0$ and $H(X)\le 0$ for $X\le X_0$. But $H(X_0)=0$, so it further suffices to show $H'(X)\ge 0$ for all $X$. But direct computation gives
    \begin{equation*}
        H'(X)=\frac{f'(X_0)[f''(X)^2-f'''(X)f'(X)]}{f'(X)^2},
    \end{equation*}
    which is nonnegative by \eqref{eqn_condition_f_3}. This finishes the proof.
\end{proof}

\subsection{End of construction}
Denote
\begin{equation*}
    m:=M/2.
\end{equation*}
We perform the above procedures for $m$ times, once at each tangent point $x_j=2j/M$, $j=1,\dots,m$, arriving at the set $T^{(m)}$. 

{\it Remark.} Compressing only $M/2$ times but not $M$ times is a new feature of this paper. On the one hand, this technical treatment is necessary (see, for example, the upper bounds of \eqref{eqn_unvn} below). On the other hand, this is also sufficient for Theorem \ref{thm_measure_bound} to hold.

We now define our required set
\begin{equation*}
    F_M:=T^{(m)}\cap \left(\left[\frac{4\log M}M,1\right]\times \R\right).
\end{equation*}
Note that this means we throw away the part of $T^{(m)}$ over $[0,x_j]$ where $j< 2\log M$. The choice of the cutoff $\frac{4\log M}M$ will be clear later in \eqref{eqn_single_thickness_bound} in the proof of Theorem \ref{thm_measure_bound}.

\subsection{Computation of translations}

Our first task is to control the sum of all translations $(u_n^{(j)},v_n^{(j)})$ that have been performed to the original curved rectangle $T^{(0)}_n$ at Steps $j=1,2,\dots,m$.

Given $n=0,1,\dots,2^M-1$, by binary expansion, we know there exist unique integers $\eps_j(n)\in \{0,1\}$, $1\le j\le M$ such that $n=\sum_{j=1}^M \eps_j(n) 2^{j-1}$. 

For convenience, we introduce the notation
\begin{equation*}
    \fj{n}{j}:=n-(n\,\,\mathrm{ mod  }\,\,2^{j-1}),
\end{equation*}
which means the ``integral part" of $n$ in $2^{j-1}\Z$. Then we note that $\eps_j(n)=1$ if and only if $u_{\fj n j}$, $v_{\fj n j}$ are defined by Step $j$.

\begin{proposition}\label{prop_translation_bound}
For $1\le j\le m$ and $n=0,1,\dots,2^M-1$, denote the partial sums 
\begin{equation}\label{eqn_partial_sum}
    U_n^{(j)}:=\sum_{i=1}^j \eps_i(n) u^{(i)}_{\fj n i},\quad V_n^{(j)}:=\sum_{i=1}^j \eps_i(n) v^{(i)}_{\fj n i}.
\end{equation}
Then we have
\begin{enumerate}
    \item \label{item:Mar_4_01} Each initial curved rectangle $T^{(0)}_n$, after $j$ steps, are exactly translated by $(U_n^{(j)},V_n^{(j)})$.
    \item \label{item:Mar_4_02} The translations satisfy
    \begin{equation}\label{eqn_partial_sum_bound}
    0\le U_n^{(j)}\lesssim \delta_0 2^{j-M},\quad |V_n^{(j)}|\lesssim  \delta_0 2^{j-M}.
\end{equation}
In particular, the total translation $(U_n^{(m)},V_n^{(m)})$ of $T_n^{(0)}$ after $m$ steps satisfies
    \begin{equation}\label{eqn_unvn}
        0\le U^{(m)}_n\lesssim\delta_0 2^{-M/2},\quad |V^{(m)}_n|\lesssim\delta_0 2^{-M/2}.
    \end{equation}
\end{enumerate}
\end{proposition}

\begin{proof}
The proof of part \eqref{item:Mar_4_01} follows by inspection, and for simplicity, we only elaborate on a typical example. For instance, if $m>100$ and $n=27$, then $T_{n}^{(0)}$ is translated according to the bottoms of $T_{27}^{(0)}$, $T_{26}^{(0)}$, $T_{24}^{(0)}$ and $T_{16}^{(0)}$ at Steps $1,2,4,5$, respectively; it remains unchanged at all other steps. Note that $\eps_j(27)=1$ if and only if $j=1,2,4,5$, whence $27-(27\,\,\mathrm{ mod }\,\,2^{j-1})=27,26,24,16$, respectively.

Part \eqref{item:Mar_4_02} follows directly from \eqref{eqn:unj_bound} and \eqref{eqn:vnj_bound}.
\end{proof}
This proposition ensures that for large $M$, the total distance of translations is tiny; in particular, it can be less than $\frac{4\log M}M$, so that the projections of the translated curved rectangles onto the $x$-axis all contain $[\frac{4\log M}M,1]$.

\subsection{Upper bound of measure}
In this subsection, we control the measure of the set $F_M$ we constructed.

\begin{theorem}\label{thm_measure_bound}
    The set $F_M$ satisfies 
    \begin{equation*}
        |F_M(\delta_0 2^{-M})|\lesssim \delta_0 M^{-2}.
    \end{equation*}
\end{theorem}
As a corollary, we obtain the lower bound for the norm of the maximal operator $R_\delta$ claimed in Theorem \ref{thm_main}.
\begin{corollary}
    The lower bound \eqref{eqn_M_lower_bound} holds.
\end{corollary}
\begin{proof}[Proof of corollary assuming Theorem \ref{thm_measure_bound}]
    Let $\delta=2^{-M}$ and take $g=1_{F_M}$ corresponding to $a_0=1,\delta_0=1$. Then the construction gives $\mathcal R_\delta g(a)\sim 1$ for every $a\in [1,2]$. Then the result follows from Theorem \ref{thm_measure_bound}.
\end{proof}

\begin{proof}[Proof of Theorem \ref{thm_measure_bound}]
It suffices to show that for each $x_0\in [\frac{4\log M}M,1]$,
\begin{equation}\label{eqn:each_slice}
    |\{y\in \R:(x_0,y)\in F_M(\delta_0 2^{-M})\}|\lesssim \delta_0 M^{-2}.
\end{equation}
Fix $j\in [2\log M,m-1]$ and assume $x_0\in [x_j,x_{j+1}]$. For each $n=0,1,\dots,2^M-1$, we write
\begin{equation}\label{eqn_n_sum}
    n=p+\sum_{i=1}^j \eps_i(n) \cdot 2^{i-1}\quad (p=0,2^j,\cdots, 2^{M}-2^j),
\end{equation}
for digits $\eps_i(n)\in \{0,1\}$.

In words, for each $j$, we group the translated curved rectangles $T_n^{(0)}+(U^{(m)}_n,V_n^{(m)})$ into $2^{M-j}$ groups, and the curved rectangles in the $p2^{-j}$-th group keep being translated together after Step $j$.

Denote by $L_n^{(j)}$ the bottom curve of $T_n^{(0)}$ after $j$ steps of translations, namely,
\begin{equation}\label{eqn_Lnj}
    L_n^{(j)}(x):=L_n^{(0)}(x-U_n^{(j)})+V_n^{(j)}.
\end{equation}
By the triangle inequality, it suffices to show that for each $p=0,2^j,\dots,2^M-2^j$, the thickness of the $p2^{-j}$-th group is $\lesssim \delta_0 M^{-2}2^{j-M}$. After $m$ steps of translations, the ``bottom curve" of this group at $x_0$ is $L^{(m)}_p(x_0)$.

Hence, to prove \eqref{eqn:each_slice}, we need to show that for each $p\le n\le p+2^{j}-1$,
\begin{equation*}
    L^{(m)}_n(x_0)-L^{(m)}_p(x_0)+(1+2)\delta_0 2^{-M}\lesssim \delta_0 M^{-2}2^{j-M}.
\end{equation*}
Here, $L^{(m)}_n(x_0)-L^{(m)}_p(x_0)$ is the distance between the bottoms, $\delta_0 2^{-M}$ is the thickness of one smallest curved rectangle, and $2\delta_0 2^{-M}$ comes from the thickening of $F_M$ by $\delta_0 2^{-M}$.

But by our choice that $j\ge 2\log M$, we have 
\begin{equation}\label{eqn_single_thickness_bound}
    2^{-M}\lesssim M^{-2}2^{j-M}.
\end{equation}
Thus our task reduces to showing
\begin{equation*}
    L^{(m)}_n(x_0)-L^{(m)}_p(x_0)\lesssim \delta_0 M^{-2}2^{j-M}, \quad \forall 0\le n\le 2^{M}-1.
\end{equation*}
We trace back to the configuration immediately after Step $j$. That is, we let 
\begin{equation*}
    x:=x_0-U^{(m)}_n+U^{(j)}_n,
\end{equation*}
which lies within $[x_{j-1},x_{j+1}]$, by \eqref{eqn_partial_sum_bound}. By the definition of $x$, we thus have
\begin{align*}
    L^{(m)}_n(x_0)-L^{(m)}_p(x_0)=L_n^{(j)}(x)-L_p^{(j)}(x),
\end{align*}
and so we will estimate $L_n^{(j)}(x)-L_p^{(j)}(x)$. By \eqref{eqn_Lnj} and \eqref{eqn_defn_an},
$$
\left\{
\begin{aligned}
    & L_n^{(j)}(x) =  a_n f(x- U_n^{(j)} ) + V_n^{(j)}, \\
    & L_{p}^{(j)}(x) =  a_p f(x- U_p^{(j)} ) + V_p^{(j)}. \\
\end{aligned}
\right.
$$
According to the definition of $U_n^{(j)}$ and $V_n^{(j)}$, we know that 
$$U_p^{(j)} = V_p^{(j)} =0. $$
We then Taylor expand $f$: 
$$f(x- U_n^{(j)}) = f(x) -f'(x) U_n^{(j)} + O(U_n^{(j)})^2, $$ 
and by \eqref{eqn_partial_sum_bound}, we have $|U_n^{(j)}|^2\lesssim \delta_0^2 2^{-2j-2M}\ll \delta_0 M^{-2}2^{j-M}$. Thus everything boils down to showing
\begin{equation}\label{eqn_00}
|a_n(f(x) - f'(x) U_n^{(j)}) + V_n^{(j)} - a_p f(x)|\lesssim \delta_0 M^{-2}2^{j-M}.
\end{equation}
We compute the left hand side using \eqref{eqn_n_sum} and \eqref{eqn_partial_sum}:
\begin{equation}\label{eqn:Mar_4_9:35}
    \text{LHS of \eqref{eqn_00}}=\sum_{i=1}^j\eps_i(n) \left(\delta_0 2^{i-1-M}f(x)-a_n f'(x)u_{\fj n i}^{(i)}+v_{\fj n i}^{(i)}\right).
\end{equation}
Using \eqref{eqn:unj_big_O} and the fact that $n-\fj n i\le 2^{i-1}$ which implies $a_n-a_{\fj n i}\le \delta_0 2^{i-1-M}$, we have 
\begin{align*}
  u_{\fj n i}^{(i)}&=\frac{\delta_0 2^{i-1-M} } {a_{\fj n i}} \frac{f'(x_i)}{f''(x_i)}+O(\delta_0^2 2^{2i-2M}) \\
  &=\frac{\delta_0 2^{i-1-M} } {a_{n}} \frac{f'(x_i)}{f''(x_i)}+O(\delta_0^2 2^{2i-2M}).
\end{align*}
Similarly, using \eqref{eqn_vnj}, we have
\begin{align*}
    v_{\fj n i}^{(i)}&=\delta_0 2^{i-1-M}\left(-f(x_i)+\frac{f'(x_i)^2}{f''(x_i)}\right)+O(\delta_0^2  2^{2i-2M}).
\end{align*}
Plugging this into \eqref{eqn:Mar_4_9:35}, we have
\begin{align*}
    &\text{LHS of \eqref{eqn_00}}\\
    &=\sum_{i=1}^j\eps_i(n) \delta_0 
 2^{i-1-M}\left(f(x)-\frac{f'(x)f'(x_i)}{f''(x_i)}+\frac{f'(x_i)^2}{f''(x_i)}-f(x_i)\right)+O(\delta_0^2  2^{2j-2M})\\
    &=\sum_{i=1}^j\eps_i(n)\delta_0  2^{i-1-M} \left(g(x)-g(x_i)\right)+O(\delta_0^2  2^{2j-2M}),
\end{align*}
where
\begin{equation*}
    g(x):=f(x)-\frac{f'(x_i)}{f''(x_i)}f'(x).
\end{equation*}
Hence, to prove \eqref{eqn_00}, it suffices to show 
\begin{equation*}
    \sum_{i=1}^j 2^{i}|g(x)-g(x_i)|\lesssim 2^j M^{-2}.
\end{equation*}
But $g'(x_i)=0$, and $g''$ exists and is bounded, since $f'''$ exists and is bounded. Thus we have $|g(x)-g(x_i)|\lesssim |x-x_i|^2$, and so we have reduced the problem to proving
\begin{equation*}
    \sum_{i=1}^j 2^i |x-x_i|^2\lesssim 2^j M^{-2}.
\end{equation*}
However, using $x\in  [x_{j-1},x_{j+1}]$ and the definition that $x_i=2i/M$ (see \eqref{eqn:defn_x_j}), this further reduces to the elementary inequality
\begin{equation*}
    \sum_{i=1}^j 2^i (j-i)^2\lesssim 2^j,
\end{equation*}
which follows from summation by parts formula. This finishes the proof.

\end{proof}

\section{Construction of zero measure Kakeya set}\label{sec_iteration}

In this section, we use a routine method to construct the curved Kakeya set $K$ with zero measure mentioned in Theorem \ref{thm_general}, based on the sets $F_M$ constructed in the previous section.

\subsection{Step 1}

Start with $T^{(0)}(a_0=1,\delta_0=1)$ as defined in \eqref{eqn_T0}. Pick a large integer $M\in 2\N$ such that we have all the results in Section \ref{sec_construction}. Denote
\begin{equation*}
    M_l:=M^l.
\end{equation*}
We then construct the set
\begin{equation*}
    K_1:=F_{M_1}=F_M,
\end{equation*}
which is a compact subset of $[\frac{4\log M}M,1]\times [-C_0,C_0]$ for some large constant $C_0=C_0(f)$. Denote $A_1:=\frac{4\log M}{M}$.

\subsection{Step 2}
Recall $K_1$ consists of $2^{M_1}$ smaller curved rectangles, which we denote by $T(1)$. Each $T(1)$ is of the form $T^{(0)}(a,2^{-M_1})+(u,v)$ for some $a\in [1,2]$, $|u|\lesssim 2^{-M_1/2}$, $|v|\lesssim 2^{-M_1/2}$ by \eqref{eqn_unvn}. Theorem \ref{thm_measure_bound} states that 
\begin{equation}\label{eqn:Mar_4_K1}
    |K_1(2^{-M_1})|\lesssim M_1^{-2}.
\end{equation}
We then apply the same procedure in Section \ref{sec_construction} to each $T(1)$, this time with $a_0$ taken to be this $a$, $M$ taken to be $M_2$, and $\delta_0$ taken to be $2^{-M_1}$. (More precisely, to perfectly fit the notation of Section \ref{sec_construction} we should first reverse the translation by $(U_n^{(m)},V_n^{(m)})$ to $T(1)$, rescale so that it is over $[0,1]$, apply the construction and then rescale back in the end.) 

Denote by $\tilde T(1)$ the configuration obtained from $T(1)$ after the aforesaid operation, and note that $\tilde T(1)$ consists of $2^{M_2}$ curved rectangles of thickness $2^{-M_1-M_2}$. By Theorem \ref{thm_measure_bound}, the $2^{-M_1-M_2}$ neighbourhood of $\tilde T(1)$ has Lebesgue measure $\lesssim 2^{-M_1} M_2^{-2}$.

Denote by $K_2$ the union of all such $2^{M_1}$ configurations $\tilde T(1)$. We have
\begin{equation}\label{eqn:Mar_4_K2}
    |K_2(2^{-M_1-M_2})|\lesssim 2^{M_1} (2^{-M_1} M_2^{-2})=M_2^{-2}.
\end{equation}
Also, by \eqref{eqn_unvn}, the translations at Step 2 is $\le C2^{-M_1-M_2/2}$, where $C$ depends only on $f$. Thus 
\begin{equation*}
    K_2(2^{-M_1-M_2})\sub K_1(2^{-M_1-M_2}+C2^{-M_1-M_2/2})\sub K_1(2^{-M_1}),
\end{equation*}
where we have assumed $M$ is large enough compared with $C$.

Lastly, we can check that the projection of $K_2$ onto the $x$-axis is equal to $[A_2,1]$ where
\begin{equation*}
    A_2:=1-\left(1-\frac{4\log M_1}{M_1}\right)\left(1-\frac{4\log M_2}{M_2}\right).
\end{equation*}

\subsection{Step 3}
We continue this process. Now $K_2$ consists of $2^{M_1+M_2}$ even smaller curved rectangles which we denote by $T(2)$. We then apply the same procedure in Section \ref{sec_construction} to each such $T(2)$, this time with $M$ taken to be $M_3$ and $\delta_0$ taken to be $2^{-M_1-M_2}$. This gives us the set $K_3$, whose projection onto the $x$-axis is equal to $[A_3,1]$ where
\begin{equation*}
    A_3:=1-\left(1-\frac{4\log M_1}{M_1}\right)\left(1-\frac{4\log M_2}{M_2}\right)\left(1-\frac{4\log M_3}{M_3}\right).
\end{equation*}
By a similar reason to \eqref{eqn:Mar_4_K1}, we have
\begin{equation*}
    |K_3(2^{-M_1-M_2-M_3})|\lesssim 2^{M_1+M_2} (2^{-M_1-M_2} M_3^{-2})=M_3^{-2}.
\end{equation*}
Also, by \eqref{eqn_unvn}, we have
\begin{equation}\label{eqn_K3K2}
    K_3(2^{-M_1-M_2-M_3})\sub K_2(2^{-M_1-M_2-M_3}+C2^{-M_1-M_2-M_3/2})\sub K_2(2^{-M_1-M_2}).
\end{equation}

\subsection{Step $l$}
Continuing this process, we obtain a nested sequence of nonempty compact sets $\tilde K_l:=K_l(2^{-M_1-\cdots-M_l})$ that satisfy
\begin{equation}\label{eqn:Mar_4_intermediate}
    \tilde K_l\sub K_{l-1}(2^{-M_1-\cdots-M_l}+C2^{-M_1-\cdots-M_{l-1}/2})\sub \tilde K_{l-1},
\end{equation}
whose measure satisfies
\begin{equation}\label{eqn:Mar_4_Kl}
    |\tilde K_l|\lesssim M_l^{-2}.
\end{equation}

We are now ready to define
\begin{equation*}
    K:=\bigcap_{l=1}^\infty \tilde K_l=\bigcap_{l=1}^\infty K_l(2^{-M_1-M_2-\cdots-M_{l}}),
\end{equation*}
which is a nonempty compact set by Cantor's intersection theorem.
\begin{theorem}
The following statements hold for $K$.
\begin{enumerate}
    \item $K$ has zero two-dimensional Lebesgue measure.
    \item $K$ contains a translation of a piece of length $\sim 1$ of the graph of $y=af(x)$ for every $1\le a\le 2$.
    \item $|K(\delta)|\lesssim (\log \delta^{-1})^{-2}$.
\end{enumerate}
\end{theorem}
\begin{proof}
\begin{enumerate}
    \item This part follows immediately from \eqref{eqn:Mar_4_Kl} by taking $l\to \infty$.
    
    \item After the $l$-th step, the projection of $K_l$ onto the $x$-axis is equal to $[A_l,1]$ where
    \begin{equation*}
        A_l:=1-\prod_{l'=1}^l\left(1-\frac{4\log M_{l'}}{M_{l'}}\right).
    \end{equation*}
    Thus it suffices to prove that $\lim_l A_l<1$, which is equivalent to proving $\sum_{l}\frac{4\log M_l}{M_l}<\infty$. But by our choice of $M_l=M^l$, this follows.
    \item Given $\delta\in (0,1)$, take the unique $l$ such that 
    \begin{equation*}
        2^{-M_1-\cdots-M_{l+1}}<\delta\le 2^{-M_1-\cdots-M_{l}}.
    \end{equation*}
    Then by \eqref{eqn_unvn}, we have
    \begin{align*}
        |K(\delta)|
        &\le |\tilde K_l(\delta)|\\
        &\le |K_{l-1}(2^{-M_1-\cdots-M_{l}}+C2^{-M_1-\cdots-M_l/2}+\delta)|\\
        &\le |K_{l-1}(2^{-M_1-\cdots-M_{l-1}})|\\
        &\lesssim M_{l-1}^{-2}.
    \end{align*}
    On the other hand, using $2^{-M_1-\cdots-M_{l+1}}<\delta$, we have $M_l=M^l\gtrsim \log \delta^{-1}$. Hence we have $M_{l-1}^{-2}\lesssim (\log \delta^{-1})^{-2}$.
\end{enumerate}
\end{proof}

\section{Appendix}
In this appendix, we provide a brief summary of the $L^p$ to $L^q$ boundedness of the maximal operator $\mathcal R_\delta$ defined in \eqref{eqn:cinematic_delta}. We assume in addition that $f$ is smooth, but without \eqref{eqn_condition_f_3}. Recall from Section \ref{sec:cinematic} that the second inequality of \eqref{eqn_condition_f_1} guarantees that the family of curves $\{(t,af(t)):0\le t\le 1\}$, $1\le a\le 2$ satisfies the cinematic curvature condition as stated in \cite{Kolasa_Wolff}.

\begin{theorem}
    In the $(1/p,1/q)$-interpolation diagram (see Figure \ref{fig:interpolation}), let $O$ be the origin, and
    \begin{align*}
    &A=\left(0,1\right), \quad B=\left(\frac 1 {3},1\right),\quad C=\left(\frac 3 {8},1\right),\quad D=\left(1,1\right),\\
    &E=\left(1,0\right),\quad F=\left(\frac 3 8,0\right),\quad G=\left(\frac 3 8,\frac 5 {16}\right),\quad H=\left(\frac 1 {3},\frac 1 3\right).
\end{align*}

\begin{figure}[h]
    \centering
    \includegraphics[width=0.74\linewidth]{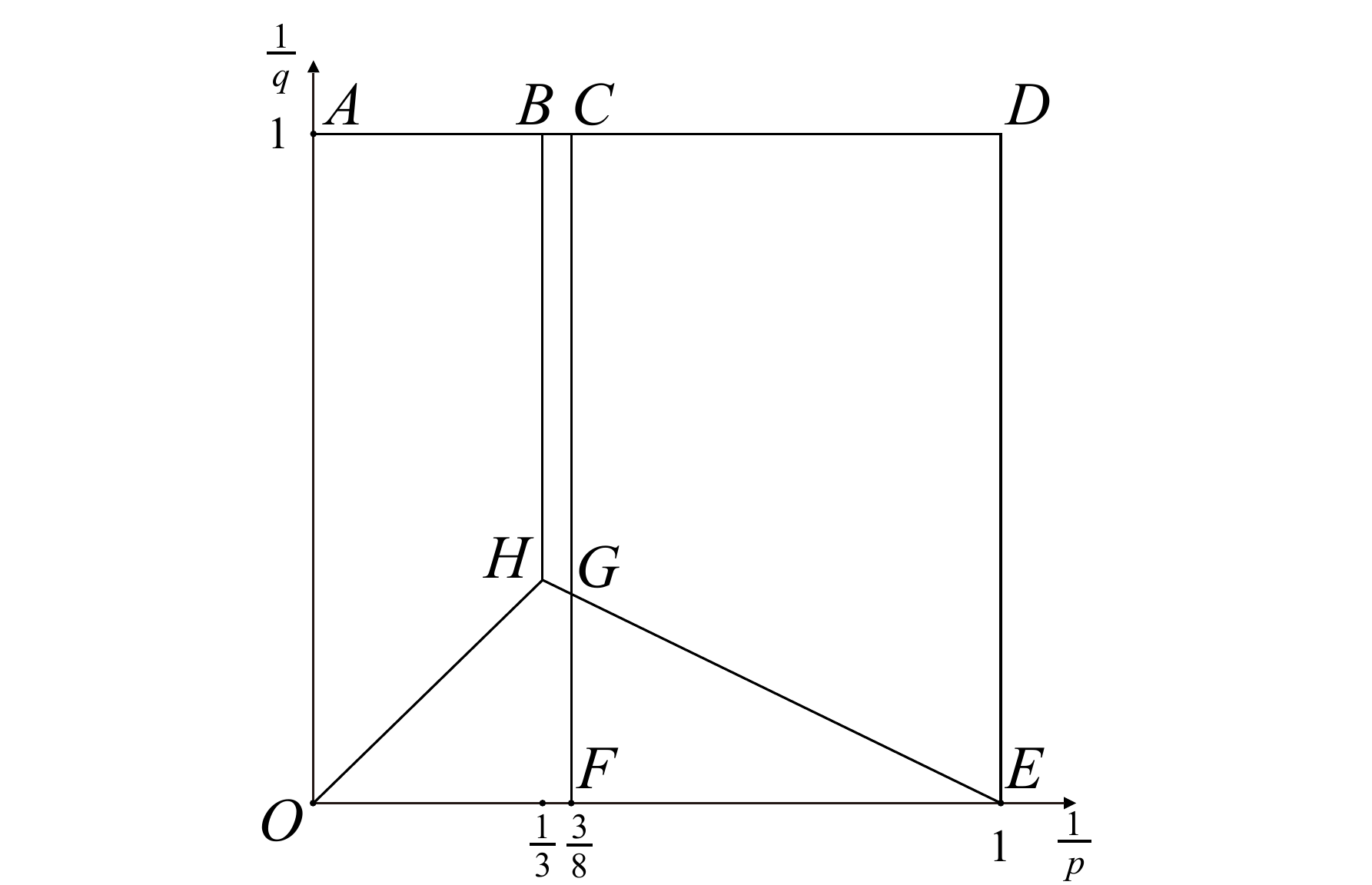}
    \vspace{-3mm}
    \caption{Interpolation diagram of $\mathcal R_\delta$}
    \label{fig:interpolation}
\end{figure}

Then we have the following estimates (see Table \ref{tab:table1}). Here all line segments and polygons below include their boundaries, and the implicit constants are allowed to depend on $p,q$ but not $\delta$.

\begin{table}[h]
\begin{tabular}{|c|c|}
\hline
 $R(p,q,\delta)$ & $\forall(1/p,1/q)\in$ \\ \hline
 $\sim 1$ & $OA$ \\ \hline
     $\begin{cases}
        \lesssim_\eps \delta^{-\eps}\\ 
        \gtrsim (\log \delta^{-1})^{2/p}
    \end{cases}$ & $OABH\backslash OA$ \\ \hline
     $\begin{cases}
        \lesssim_\eps \delta^{-\eps+\frac 1 2-\frac 3 {2p}}\\ 
        \gtrsim \delta^{\frac 1 2-\frac 3 {2p}}
    \end{cases}$ &  $BCGH\backslash BH$ \\ \hline
$\sim\delta^{\frac 1 2-\frac 3 {2p}}$ & $CDEG\backslash CG$ \\ \hline
$\sim \delta^{\frac 1 q-\frac 1 {p}}$ & $EFG\backslash FG$ \\ \hline
$\begin{cases}
        \lesssim_\eps \delta^{-\eps+\frac 1 q-\frac 1 {p}}\\ 
        \gtrsim \delta^{\frac 1 q-\frac 1 {p}}
    \end{cases}$ & $OFGH\backslash OF$ \\  \hline
$\sim \delta^{-1/p}$ & $OE$ \\ \hline
\end{tabular}
\caption{$L^p$ to $L^q$ boundedness of the maximal operator $\mathcal R_\delta$\looseness=-1}\label{tab:table1}
\end{table}
\end{theorem}

\begin{proof}
    We first come to the case $p=\infty$, namely, $(1/p,1/q)\in OA$. Here the upper bound is trivial, and the lower bound follows from taking $f=1_{B(0,1)}$.

    For the case $q=\infty$, namely, $(1/p,1/q)\in OE$, the upper bound follows from interpolating between $p=1$ and $p=\infty$, and the lower bound follows from taking $f=1_{S}$ where
    \begin{equation}\label{eqn_defn_S}
        S=\{(x,ax^2):x\in [0,1],a\in [1,1+\delta]\}.
    \end{equation}
    
    When $(1/p,1/q)\in OABH\backslash OA$, to prove the upper bound, by H\"older's inequality, it suffices to prove the bound on $OH$ only. By \cite{Zahl_nonalgebraic}, it holds at $H=(1/3,1/3)$. Then this follows from interpolating between $(0,0)$ and $(1/3,1/3)$. The lower bound follows from Theorem \ref{thm_general}.

    When $(1/p,1/q)\in CDEG\backslash CG$, the upper bound follows from \cite{Kolasa_Wolff}. The lower bound follows from taking $f=1_T$, where $T$ is a rectangle of dimensions $\delta^{1/2}\times \delta$.

    When $(1/p,1/q)\in EFG\backslash FG$, the upper bound follows from interpolating between the points $(1/p,0)$ and $(\frac 1 p,\frac 1 2-\frac 1 {2p})$. The lower bound follows from taking $f=1_S$, where $S$ is given by \eqref{eqn_defn_S}.

    When $(1/p,1/q)\in BCGH\backslash BH$, the upper bound follows from interpolating between $BH$ and $CG$. The lower bound follows from taking $f=1_T$, where $T$ is a rectangle of dimensions $\delta^{1/2}\times \delta$.

    When $(1/p,1/q)\in OFGH\backslash OF$, the upper bound follows from interpolating between $OF$ and the segments $OH,HG$. The lower bound follows from taking $f=1_S$, where $S$ is given by \eqref{eqn_defn_S}.
\end{proof}

\subsection{Open problems}
It is natural to ask whether we can replace the $\eps$-loss on the upper bounds of $R(p,q,\delta)$ with a logarithmic loss (when we have a logarithmic lower bound), or with a constant loss (when we have a lower bound of the form $\delta^{\alpha}$, $\alpha>0$, without $\eps$-loss). More precisely, one can consider the following problems.
\begin{enumerate}
    \item Improving the upper bound of $R(3,3,\delta)$. One may work through all the details of \cite{SchlagJFA} and \cite{Wolff2000} to see if each time a loss of the form $\delta^{-\eps}$ can be upgraded to just $\log \delta^{-1}$; if so, then we can improve the bound to $(\log \delta^{-1})^{O(1)}$. However, improving towards the lower bound $(\log \delta^{-1})^{2/p}$ shown in Theorem \ref{thm_general} seems to require a new method.
    \item Removing the $\eps$-loss in the range $\frac 1 3<\frac 1 p\le\frac 3 8$. The condition $p<8/3$ is needed in the combinatorial argument in \cite{Kolasa_Wolff}, but the authors are not aware of any counterexample showing its necessity.

\end{enumerate}

 \bibliographystyle{alpha}
 \bibliography{sample}

\end{document}